\documentclass{amsart}
\usepackage{amsfonts,amssymb,amsmath,amsthm}
\usepackage{url}
\usepackage{enumerate}
\usepackage{bbm}
\usepackage{amssymb}

\newtheorem{thrm}{Theorem}[section]
\newtheorem{lem}[thrm]{Lemma}

\theoremstyle{definition}

\numberwithin{equation}{section}

\newcommand{\N}{\ensuremath{{\mathbb N}}}

\newcommand{\R}{\ensuremath{{\mathbb R}}}
\newcommand{\e}{\varepsilon}
\newcommand{\AveP}{\underset{\pi}{\mbox{Ave}}}

\newcommand{\abs}[1]{\left\lvert#1 \right\rvert}
\newcommand{\norm}[1]{\left \lVert#1 \right\rVert}
\newcommand{\skp}[1]{\left<#1\right>}

\author{Joscha Prochno}

\address[Joscha Prochno]{Department of Mathematical and Statistical Sciences\\
University of Alberta\\
505 Central Academic Building\\
Edmonton T6G 2G1\\
Canada}
\email{prochno@ualberta.ca}
\thanks{~}

\begin{document}

\title[Combinatorial Approach to Musielak-Orlicz Spaces]{Combinatorial Approach to Musielak-Orlicz Spaces}

\begin{abstract}
In this paper we show that, using combinatorial inequalities and Matrix-Averages, we can generate Musielak-Orlicz spaces, {\it i.e.}, we prove that $\AveP \max\limits_{1 \leq i \leq n} \abs{x_i y_{i\pi(i)}} \sim \norm{x}_{\Sigma M_i}$, where the Orlicz functions $M_1,\ldots,M_n$ depend on the matrix $(y_{ij})_{i,j=1}^n$. We also provide an approximation result for Musielak-Orlicz norms which already in the case of Orlicz spaces turned out to be very useful.
\end{abstract}
\maketitle

\section{Introduction} \label{sect1}

Understanding the structure of the classical Banach space $L_1$ is an important goal of Banach Space Theory, since this space naturally appears in various areas of mathmatics, {\it e.g.}, Functional Analysis, Harmonic Analysis and Probability Theory. One way to do this is to study the ``local'' properties of a given space, {\it i.e.}, the finite-dimensional subspaces, which on the other hand bears information about the ``global'' structure.

In \cite{key-K-S1} and \cite{key-K-S2}, Kwapie\'n and Sch\"utt proved several combinatorial and probabilistic inequalities and used them to study invariants of Banach spaces and finite-dimensional subspaces of $L_1$. Among other things, they considered for $x,y\in\R^n$ 
  $$
    \AveP \max_{1\leq i \leq n}\abs{x_i y_{\pi(i)}},
  $$
and gave the order of the combinatorial expression in terms of an Orlicz norm of the vector $x$. In fact, this is not only a main ingredient to prove that every finite-dimensional symmetric subspace of $L_1$ is $C$-isomorphic to an average of Orlicz spaces (see \cite{key-K-S1}), but also to show that an Orlicz space with a $2$-concave Orlicz function is isomorphic to a subspace of $L_1$ (see \cite{key-Sch1}). 
Here, we are going to generalize these results and consider combinatorial Matrix-Averages, {\it i.e.}, 
  \begin{equation}\label{EQU matrix average}
    \AveP \max_{1 \leq i \leq n} \abs{x_i y_{i\pi(i)}}, 
  \end{equation}
with $x\in\R^n$, $y\in\R^{n\times n}$, and express their order in terms of Musielak-Orlicz norms. The new approach is to average over matrices instead of just vectors. This corresponds to the idea of considering random variables that are not necessary identically distributed. In fact, using this idea one can also generalize the results from \cite{key-GLSW} to the case of Musielak-Orlicz spaces. 
We prove that
  $$
    C_1 \norm{x}_{\Sigma M^*_i} \leq \AveP \max_{1 \leq i \leq n} \abs{x_i y_{i\pi(i)}} \leq C_2 \norm{x}_{\Sigma M^*_i},
  $$
where $C_1,C_2>0$ are absolute constants and the dual Orlicz functions $M^*_1,\ldots,M^*_n$ depend on $y\in\R^{n\times n}$. In Section $4$, we also provide the converse result, {\it i.e.}, given Orlicz functions $M_1,\ldots,M_n$, we show which matrix $y\in\R^{n\times n}$ yields the equivalence of (\ref{EQU matrix average}) to the corresponding Musielak-Orlicz norm $\norm{\cdot}_{\Sigma M_i^*}$.
In the last section we prove an approximation results for Musielak-Orlicz norms. In applications, a corresponding results for Orlicz norms turned out to be quite fruitful and simplified calculations (see \cite{key-GLSW}).

However, these Musielak-Orlicz norms are generalized Orlicz norms in the sense that one considers a different Orlicz function in each component. Since one can use the combinatorial results in \cite{key-K-S1}, \cite{key-K-S2} to study embeddings of Orlicz and Lorentz spaces into $L_1$ (see \cite{key-Pro}, \cite{key-Sch1}, \cite{key-Sch2}), the results we obtain can be seen as a point of departure to obtain embedding theorems for more general classes of finite-dimensional Banach spaces into $L_1$, {\it e.g.}, Musielak-Orlicz spaces. This, on the other hand, is crucial to extend the understanding of the geometric properties of $L_1$.

\section{Preliminaries}

A convex function $M:[0,\infty) \to [0,\infty)$ with $M(0)=0$ and $M(t)>0$ for $t>0$ is called an Orlicz function. 
Given an Orlicz function $M$ we define its conjugate or dual function $M^*$ by the Legendre-Transform
  $$
    M^*(x) = \sup_{t\in[0,\infty)}(xt-M(t)).
  $$
 Again, $M^*$ is an Orlicz function and $M^{**}=M$, which yields that an Orlicz function $M$ is uniquely determined by the dual function $M^*$. For instance, taking $M(t)=\frac{1}{p}t^p$, $p\geq 1$, the dual function is given by $M^*(t)=\frac{1}{p^*}t^{p^*}$ with $\frac{1}{p^*}+\frac{1}{p}=1$. 
We define the $n$-dimensional Orlicz space $\ell^n_M$ to be $\R^n$ equipped with the norm
  $$
    \norm{x}_M = \inf\left\{ \rho>0 : \sum_{i=1}^n M\left(\frac{\abs{x_i}}{\rho}\right) \leq 1  \right\}.
  $$
Notice that to each decreasing sequence $y_1\geq \ldots \geq y_n > 0$ there corresponds an Orlicz function $M:=M_y$ via 
  $$
    M\left(\sum_{i=1}^k y_i\right) = \frac{k}{n}, ~ k=1,\ldots,n,
  $$ 
and where the function $M$ is extended linearly between the given values.\\  
Let $M_1,\ldots,M_n$ be Orlicz functions. We define the $n$-dimensional Musielak-Orlicz space $\ell^n_{\Sigma M_i}$ to be the space $\R^n$ equipped with the norm
  $$
    \norm{x}_{\Sigma M_i} = \inf \left\{ \rho>0 : \sum_{i=1}^n M_i\left( \frac{\abs{x_i}}{\rho} \right) \leq 1 \right\}.
  $$
These spaces can be considered as generalized Orlicz spaces. One can easily show (see \cite{key-Pro}, Lemma 7.3), using Young's inequality, that the norm of the dual space $(\ell^n_{\Sigma M_i})^*$ is equivalent to
  $$
    \norm{x}_{\Sigma M_i^*} = \inf \left\{ \rho>0 : \sum_{i=1}^n M_i^*\left( \frac{\abs{x_i}}{\rho} \right) \leq 1 \right\},
  $$
which is the analog result as for the classical Orlicz spaces. To be more precise, we have $\norm{\cdot}_{\Sigma M_i^*} \leq \norm{\cdot}_{\left( \Sigma M_i\right)^*} \leq 2\norm{\cdot}_{\Sigma M_i^*}$. A more detailed and thorough introduction to Orlicz spaces can be found in \cite{key-KR} and \cite{key-RR}.

We will use the notation $a\sim b$ to express that there exist two positive absolute constants $c_1, c_2$ such that $c_1a\leq b\leq c_2 a$. The letters 
$c,C,C_1,C_2,\ldots$ will denote positive absolute constants, whose value may change from line to line. By $k,m,n$ we will denote natural numbers.

In the following, $\pi$ is a permutation of $\{1,\ldots,n\}$ and we write $\AveP$ to denote the average over all permutations in the group 
$\mathfrak{S}_n$, {\it i.e.}, $\AveP:=\frac{1}{n!}\sum_{\pi\in\mathfrak{S}_n}$.\\

We need the following result from \cite{key-K-S1}.

\begin{thrm} [\cite{key-K-S1} Theorem 1.1] \label{THM_Ausgangspunkt}
  Let $n\in\N$ and $y=(y_{ij})_{i,j=1}^n\in\R^{n\times n}$ be a real $n\times n$ matrix. Then
    $$
      \frac{1}{2n} \sum_{k=1}^n s(k) \leq \AveP \max_{1 \leq i \leq n} \abs{y_{i\pi(i)}}
      \leq \frac{1}{n} \sum_{k=1}^n s(k),
    $$
  where $s(k)$, $k=1,\ldots,n^2$, is the decreasing rearrangement of $\abs{y_{ij}}$, $i,j=1,\ldots,n$.
\end{thrm}
  
\section{Combinatorial Generation of Musielak-Orlicz Spaces}

We will prove that a Matrix-Average, in fact, yields a Musielak-Orlicz norm. Following \cite{key-K-S1}, we start with a structural lemma.

\begin{lem}\label{LEM_MusielakOrlicz}
  Let $y=(y_{ij})_{i,j=1}^n\in\R^{n\times n}$ be a real $n\times n$ matrix with $y_{i1} \geq \ldots \geq y_{in}>0$ and
  $\sum_{j=1}^n y_{ij}=1$ for all
  $i=1,\ldots,n$. Let $M_i$, $i=1,\ldots,n$, be convex functions with
    \begin{equation}\label{EQU_Mi}
      M_i \left( \sum_{j=1}^k y_{ij} \right) = \frac{k}{n}, ~k=1,\ldots,n.
    \end{equation}
  Furthermore, let
    $$
      B_{\Sigma M_i} = \left\{ x\in\R^n \,:\, \sum_{i=1}^n M_i(\abs{x_i}) \leq 1 \right\}
    $$
  and
    $$
      B=\hbox{convexhull} \left\{ \left(\e_i \sum_{j=1}^{\ell_i} y_{ij} \right)_{i=1}^n \,:\, \sum_{i=1}^n\ell_i \leq n
      , \e_i=\pm 1, i=1,\ldots,n \right\}.
    $$
  Then, we have
    $$
      B \subset B_{\Sigma M_i} \subset 3 B.
    $$
\end{lem}

\begin{proof}
  We start with the left inclusion:\\
  We have
    $$
      \sum_{i=1}^n M_i \left( \abs{\e_i \sum_{j=1}^{\ell_i} y_{ij}} \right)
      = \sum_{i=1}^n M_i \left( \sum_{j=1}^{\ell_i} y_{ij} \right)
      = \sum_{i=1}^n \frac{\ell_i}{n} \leq 1.
    $$
  Therefore, $B \subset B_{\Sigma M_i}$.\\
  Now the right inclusion:\\
  W.l.o.g. let
    $$
      \sum_{i=1}^n M_i (\abs{x_i}) = 1,
    $$
  {\it i.e.}, $x\in B_{\Sigma M_i}$ and $x_1 \geq \ldots \geq x_n \geq 0$. Furthermore, let $J,I \subset \{1,\ldots,n\}$ indexsets with 
  $I\cap J= \emptyset$ s.t.
    $$
      x = x_J + x_I, ~~~x_J,x_I\in\R^n,
    $$
  where we choose $J$ s.t.
    $$
      M_i(x_i) > \frac{1}{n} ~~~\hbox{for all $i\in J$}
    $$
  and $I$ s.t.
    $$
      M_i(x_i) \leq \frac{1}{n} ~~~\hbox{for all $i\in I$}.
    $$
  Let $\abs{J}=r$ and thus $\abs{I}=n-r$. We complete the vectors $x_J$ and $x_I$ in the other components with zeros. We disassemble $x$ in two
  vectors such that the associated Orlicz functions $M_i$ are greater $1/n$ and on the other segment less or equal to $1/n$. By our requirement 
  we have 
    $$
      M_i(y_{i1}) = \frac{1}{n} ~~~\hbox{for all $i=1,\ldots,n$.}
    $$
  Therefore, $x_I \leq (y_{11},\ldots,y_{n1})$, since $M_i(x_i) \leq \frac{1}{n} = M_i(y_{i1})$ for all $i\in I$. We have
  $(y_{11},\ldots,y_{n1})\in B$, which follows immediately from the choice of $\ell_i=1, \e_i=1$ for all $i=1,\ldots,n$, and therefore finally $x_I\in B$.
  It is left to show that $x_J \in 2B$. For each $i\in J$ there exists a $k_i \geq 1$ with
    \begin{equation}\label{EQU_Mi_zwischen_ki_durch_n}
      \frac{k_i}{n} \leq M_i(x_i) \leq \frac{k_i+1}{n}.
    \end{equation}
  Summing up all $i\in J$, we obtain by (\ref{EQU_Mi}) and (\ref{EQU_Mi_zwischen_ki_durch_n})
    $$
      \sum_{i\in J} \frac{k_i}{n} \stackrel{(\ref{EQU_Mi})}{=} \sum_{i\in J} M_i \left( \sum_{i=1}^{k_i} y_{ij} \right)
      \stackrel{(\ref{EQU_Mi_zwischen_ki_durch_n})}{\leq} \sum_{i\in J} M_i(x_i) \leq 1.
    $$
  Now, let $z_J\in\R^n$ be the vector with the entries $\sum_{j=1}^{k_i} y_{ij}$ at the points $i\in J$ and zeros elsewhere.
  Then, we have $z_J\in B$, because $\sum_{i\in J} k_i \leq n$. Let $w_J\in\R^n$ be the vector with the entries $\sum_{j=1}^{k_i+1} y_{ij}$ at the
  points $i\in J$ and zeros elsewhere. We have $2z_J \geq w_J$, because $y_{ij}$ is decreasing in $j$ and therefore
  $y_{ik_i+1}$ can be estimated by $\sum_{j=1}^{k_i}y_{ij}$. Furthermore, we have for all $i\in J$
    $$
      \sum_{j=1}^{k_i+1} y_{ij} \geq x_i,
    $$
  since
    $$
      M_i(x_i) \stackrel{(\ref{EQU_Mi_zwischen_ki_durch_n})}{\leq} \frac{k_i+1}{n}
      = M_i \left( \sum_{j=1}^{k_i+1} y_{ij} \right) ~~~\hbox{for all $i\in J$}.
    $$
  Hence, $2z_J \geq x_J$ and thus $x_J \in 2B$. Altogether, we obtain
    $$
      x=x_J + x_I \in 3B.
    $$
\end{proof}

Note that the condition $\sum_{j=1}^n y_{ij}=1$ is just a matter of normalization, so that we have normalized Orlicz functions with $M_i(1)=1$, and therefore can be omitted. In addition, replacing the conditions (\ref{EQU_Mi}) by
  $$
    M_i^* \left( \sum_{j=1}^k y_{ij} \right) = \frac{k}{n}, ~k=1,\ldots,n,
  $$
yields the result for the dual balls. However, from this lemma we can deduce that our combinatorial expression generates a Musielak-Orlicz norm.

\begin{thrm}\label{THM_Erzeugung_Musielak_Orlicz_Norm}
  Let $y=(y_{ij})_{i,j=1}^n\in\R^{n\times n}$. Let the assumptions be as in Lemma \ref{LEM_MusielakOrlicz}. Then, for every $x\in\R^n$,
    $$
      \frac{1}{6n} \norm{x}_{\Sigma M^*_i} \leq
      \AveP \max_{1 \leq i \leq n} \abs{x_i y_{i\pi(i)}} \leq \frac{2}{n} \norm{x}_{\Sigma M^*_i},
    $$
   where $M_i$, $i=1,\ldots,n$ are given by formula (\ref{EQU_Mi}).
\end{thrm}
\begin{proof}
 By Theorem \ref{THM_Ausgangspunkt}
    $$
      \frac{1}{2n} \sum_{k=1}^n s(k) \leq \AveP \max_{1 \leq i \leq n} \abs{x_i y_{i\pi(i)}} \leq \frac{1}{n} \sum_{k=1}^n s(k),
    $$
  where $s(k)$, $k=1,\ldots,n^2$, is the decreasing rearrangement of $\abs{x_i y_{ij}}$, $i,j=1,\ldots,n$.
  Rewriting the expression gives
    $$
      \sum_{k=1}^n s(k) = \sum_{i=1}^n \sum_{j=1}^{\ell_i} x_i y_{ij} = \sum_{i=1}^n x_i \sum_{j=1}^{\ell_i} y_{ij},
    $$
  where $\ell_i$, $i=1,\ldots,n$ are chosen to maximize the upper sum and satisfy $\sum_{i=1}^n \ell_i \leq n$. We have
    $$
      \sum_{i=1}^n x_i \sum_{j=1}^{\ell_i} y_{ij} = \skp{x,\left(\sum_{j=1}^{\ell_i} y_{ij}\right)_{i=1}^n}.
    $$
  Now, taking the supremum over all $z\in B_{\Sigma M_i}$ instead of the supremum over all elements of $B$, and using the fact that by Lemma \ref{LEM_MusielakOrlicz} $B \subset B_{\Sigma M_i}$, we get
   $$
      \AveP \max_{1 \leq i \leq n} \abs{x_i y_{i\pi(i)}} \leq \frac{1}{n} \norm{x}_{\left(\Sigma M_i\right)^*}.
   $$
 As mentioned above, we have that $\norm{\cdot}_{\left(\Sigma M_i\right)^*} \leq 2 \norm{\cdot}_{\Sigma M_i^*}$ and hence
   $$
      \AveP \max_{1 \leq i \leq n} \abs{x_i y_{i\pi(i)}} \leq \frac{2}{n} \norm{x}_{\Sigma M_i^*}.
   $$
 Similarly, now using the fact that by Lemma \ref{LEM_MusielakOrlicz} $\frac{1}{3} B_{\Sigma M_i} \subset B$ and that $\norm{\cdot}_{\Sigma M_i^*} \leq \norm{\cdot}_{\left(\Sigma M_i \right)^*}$, we obtain
   $$
       \AveP \max_{1 \leq i \leq n} \abs{x_i y_{i\pi(i)}} \geq \frac{1}{6n} \norm{x}_{\Sigma M_i^*}.
   $$
\end{proof}

If we choose a slightly different normalization as in the beginning, we obtain the following version of the theorem.

\begin{thrm} \label{THM Erzeugung von Musielak-Orlicz-Normen ueber Permutationen}
  Let $y=(y_{ij})_{i,j=1}^n$ be a real $n\times n$ matrix with $y_{i1} \geq \ldots \geq y_{in}$, $i=1,\ldots,n$. Let $M_i$, $i=1,\ldots,n$, be
  Orlicz functions with
    \begin{equation}\label{EQU_Mi_2}
      M_i \left( \frac{1}{n} \sum_{j=1}^k y_{ij} \right) = \frac{k}{n}, ~k=1,\ldots,n.
    \end{equation} Then, for every $x\in\R^n$,
    $$
      \frac{1}{6} \norm{x}_{\Sigma M^*_i} \leq
      \AveP \max_{1 \leq i \leq n} \abs{x_i y_{i\pi(i)}} \leq 2 \norm{x}_{\Sigma M^*_i}.
    $$
\end{thrm}

Again, if we assume 
$$
  M_i^* \left( \frac{1}{n} \sum_{j=1}^k y_{ij} \right) = \frac{k}{n}, ~~~k=1,\ldots,n.
$$
instead of condition (\ref{EQU_Mi_2}), we obtain
 $$
   \AveP \max_{1 \leq i \leq n} \abs{x_i y_{i\pi(i)}} \sim \norm{x}_{\Sigma M_i}.
 $$

 \section{The Converse Result}
 
We will now prove a converse to Theorem \ref{THM Erzeugung von Musielak-Orlicz-Normen ueber Permutationen}, {\it i.e.}, given a Musielak-Orlicz norm, and therefore Orlicz functions $M_i$, $i=1,\ldots,n$, we show how to choose the matrix $y=(y_{ij})_{i,j=1}^n$ to generate the given Musielak-Orlicz-Norm $\norm{\cdot}_{\Sigma M^*_i}$.

\begin{thrm}
  Let $n\in\N$ and let $M_i$, $i=1,\ldots,n$, be Orlicz functions. Then
    \begin{eqnarray*}
      C_1 \norm{x}_{\Sigma M^*_i} & \leq &
      \AveP \max_{1 \leq i \leq n} \abs{x_i \cdot n\cdot\left(M_i^{-1}\left(\frac{\pi(i)}{n}\right) -
      M_i^{-1}\left(\frac{\pi(i)-1}{n}\right) \right)} \\
      & \leq & C_2 \norm{x}_{\Sigma M^*_i},
    \end{eqnarray*}
  where $C_1,C_2>0$ are absolute constants.
\end{thrm}
\begin{proof}
  Let's consider an Orlicz function $M_i$ for a fixed $i\in \{1,\ldots,n\}$. We approximate this function by a function which is affine 
  between the given values $\frac{1}{n}, \frac{2}{n}, \ldots, \frac{n-1}{n},1$. The appropriate inverse images of the defining values are
    $$
      M_i^{-1} \left( \frac{j}{n} \right), ~j=1,\ldots,n.
    $$
  Now we choose
    $$
      y_{ij} = M_i^{-1}\left(\frac{j}{n}\right) -
      M_i^{-1}\left(\frac{j-1}{n}\right), ~j=1,\ldots,n.
    $$
  The vector $(y_{ij})_{j=1}^n\in\R^n$ generates the Orlicz function $M_i$ in the 'classical sense'. The matrix
  $y=(y_{ij})_{i,j=1}^n$ fulfills the conditions of Theorem \ref{THM_Erzeugung_Musielak_Orlicz_Norm}. Using 
  Theorem \ref{THM_Erzeugung_Musielak_Orlicz_Norm}, we finish the proof.
\end{proof}

Notice that using $M_i^*$, $i=1,\ldots,n$ to define the matrix $y=(y_{ij})_{i,j=1}^n$ yields the Musielak-Orlicz norm $\norm{\cdot}_{\Sigma M_i}$.

\section{Approximation of Musielak-Orlicz Norms}

It turned out to be useful to approximate Orlicz norms by a different norm and work with this expressions instead (see \cite{key-GLSW}). We will provide a corresponding result for Musielak-Orlicz norms.

Let $n,N\in\N$ with $n\leq N$. For a matrix $a\in\R^{n\times N}$ with $a_{i1} \geq \ldots \geq a_{iN}>0$, $i=1,\ldots,n$, 
we define a norm on $\R^n$ by
  $$
    \norm{x}_a = \max_{\sum\limits_{i=1}^n \ell_i\leq N} \sum_{i=1}^n \left( \sum_{j=1}^{\ell_i}a_{ij} \right) \abs{x_i},~x\in\R^n.
  $$
We will show that this norm is equivalent to a Musielak-Orlicz norm, which generalizes Lemma 2.4 in \cite{key-K-S2}.  

\begin{lem}
  Let $n,N\in\N$ and $n\leq N$. Furthermore, let $a\in\R^{n\times N}$ such that $a_{i,1}\geq \ldots \geq a_{i,N}>0$ and 
  $ \sum_{j=1}^N a_{i,j}=1$ for all $i=1,\ldots,n$. Let $M_i$, $i=1,\ldots,n$ be Orlicz functions such that for all $m=1,\ldots,N$
    \begin{equation}\label{EQU sum Mi star}
      M_i^{*}\left( \sum_{j=1}^{m}a_{i,j} \right)=\frac{m}{N}.
    \end{equation}
  Then, for all $x\in\R^n$,
    $$
      \frac{1}{2}\norm{x}_a \leq \norm{x}_{\Sigma M_i} \leq 2 \norm{x}_a.
    $$  
\end{lem}
\begin{proof}
  We start with the second inequality. Let $|||\cdot|||$ be the dual norm of $\norm{\cdot}_{\Sigma M^*_i}$. Then, for all $x\in\R^n$,
    $$
      \norm{x}_{\Sigma M_i} \leq |||x||| \leq 2 \norm{x}_{\Sigma M_i}.
    $$
  Now, consider $x\in\R^n$ with $x_1 \geq \ldots \geq x_n > 0$ and $\sum_{i=1}^n M^*_i(x_i)=1$, {\it i.e.}, $x\in B_{\Sigma M^*_i}$. 
  For each $i=1,\ldots,n$ there exist $\ell_i\in\{1,\ldots,N\}$ so that 
    \begin{equation} \label{EQU x_i a_ij}
      \sum_{j=1}^{\ell_i} a_{i,j} \leq x_i \leq \sum_{j=1}^{\ell_i+1} a_{i,j}.
    \end{equation}
  Since for each $i=1,\ldots,n$ the sequence $a_{i,j}$ is arranged in a decreasing order
    $$
      x_i \leq \sum_{j=1}^{\ell_i} a_{i,j} + a_{i,\ell_i+1} \leq \sum_{j=1}^{\ell_i} a_{i,j} + a_{i,1}. 
    $$ 
  We are going to prove that $(a_{i,1})_{i=1}^n$ and $\left( \sum_{j=1}^{\ell_i} a_{i,j} \right)_{i=1}^n$ are in $(B_a)^*$, 
  because then $x\in 2(B_{a})^*$ and therefore $B_{\Sigma M^*_i} \subseteq 2 (B_{a})^*$, where we denote by $B_a$ the closed unit ball with respect to the norm $\norm{\cdot}_a$. We have
    $$
      (B_{a})^* = \left\{ y \in\R^n | \forall x\in B_{a}: \skp{x,y} \leq 1 \right\}.
    $$
  Let $y \in B_{a}$, {\it i.e.},
    $$
      \max_{\sum\limits_{i=1}^n \ell_i\leq N} \sum_{i=1}^n \left( \sum_{j=1}^{\ell_i}a_{i,j} \right) \abs{y_i} \leq 1.
    $$
  Define $\tilde \ell_i=1$ for all $i=1,\ldots, n$. Then, $\sum_{i=1}^n \tilde \ell_i \leq N$ and therefore
    $$
      \skp{(a_{i,1})_{i=1}^n,y} = \sum_{i=1}^n \left(\sum_{j=1}^{\tilde \ell_i}a_{i,j}\right)y_i  
      \leq \max_{\sum\limits_{i=1}^n \ell_i\leq N} \sum_{i=1}^n \left( \sum_{j=1}^{\ell_i}a_{i,j} \right) \abs{y_i} \leq 1. 
    $$
  Thus, $(a_{i,1})_{i=1}^n \in (B_{a})^*$. Furthermore, by (\ref{EQU x_i a_ij})
    $$
      1 = \sum_{i=1}^n M^*_i(x_i) \geq \sum_{i=1}^n M^*_i\left(\sum_{j=1}^{\ell_i} a_{i,j}\right) = \sum_{i=1}^n \frac{\ell_i}{N},
    $$
  and therefore
    $$
      \sum_{i=1}^n \ell_i \leq N.
    $$ 
  Hence
    $$
      \skp{\left( \sum_{j=1}^{\ell_i} a_{i,j} \right)_{i=1}^n,y} 
      \leq \max_{\sum\limits_{i=1}^n \ell_i\leq N} \sum_{i=1}^n \left( \sum_{j=1}^{\ell_i}a_{i,j} \right) \abs{y_i} \leq 1. 
    $$
  So we have
    $$
      \left( \sum_{j=1}^{\ell_i} a_{i,j} \right)_{i=1}^n \in (B_{a})^*,
    $$
  and thus, $B_{\Sigma M^*_i} \subseteq 2 (B_{a})^*$. Hence,
    $$
      \frac{1}{2} \norm{x}_{\Sigma M_i} \leq \frac{1}{2} |||x||| \leq \norm{x}_a.
    $$
 Let us now prove the first inequality. Notice that 
   $$
     (B_a)^* = \hbox{convexhull} \left\{ \left(\e_i \sum_{j=1}^{k_i} a_{i,j} \right)_{i=1}^n \,:\, \sum_{i=1}^n k_i \leq N
      , \e_i=\pm 1, i=1,\ldots,n \right\}.
   $$
Hence, from equation (\ref{EQU sum Mi star}) it follows that
  $$
    \sum_{i=1}^n M_i^* \left( \abs{\e_i \sum_{j=1}^{k_i} a_{i,j}} \right) =  \sum_{i=1}^n \frac{k_i}{N} \leq 1,
  $$
since $\sum_{i=1}^n k_i \leq N$. Therefore, $(B_a)^* \subset B_{\Sigma M_i^*}$ and by duality, $B_{|||\cdot|||} = B_{(\Sigma M_i^*)^*} \subset B_a$. Since $|||\cdot||| \leq 2 \norm{\cdot}_{\Sigma M_i}$, we obtain for any $x\in\R^n$
  $$
    \frac{1}{2} \norm{x}_a \leq \norm{x}_{\Sigma M_i}.
  $$
Altogether this yields
  $$
     \frac{1}{2}\norm{x}_a \leq \norm{x}_{\Sigma M_i} \leq 2 \norm{x}_a,
  $$
for all $x\in\R^n$.      
\end{proof}  

Again, the condition $ \sum_{j=1}^N a_{i,j}=1$ is just a matter of normalization so we obtain normalized Orlicz functions and can be omitted. \\


\bibliographystyle{amsplain}

\end{document}